\documentclass[a4paper, 9pt]{article}
\usepackage{graphics}
\usepackage[dvips]{color}

\usepackage{lscape}
\usepackage{latexsym}
\usepackage{amsmath,verbatim}
\usepackage{graphics}
\usepackage{amsthm}
\usepackage{amssymb}
\usepackage[mathscr]{euscript}
\usepackage{yfonts}
\usepackage{makeidx}
\usepackage{stmaryrd}
\usepackage{multicol}
\usepackage{bm}
\usepackage[all]{xy}
\numberwithin{equation}{section}
\newtheorem{thm}{Theorem}[section]
\newtheorem{prop}[thm]{Proposition}

\newtheorem{cor}[thm]{Corollary}
{\bf}{\it}

\newtheorem{fthm}{Theorem}{\bf}{\it}
{\bf}{\it}
{\bf}{\it}
\theoremstyle{defn}
\newtheorem{defn}[thm]{Definition}

{\bf}{\rm}

\theoremstyle{rem}

\newtheorem{rem}[thm]{Remark}
{\bf}{\it}

\newtheorem{definition and corollary}[thm]{Definition and Corollary}

{\it}{\rm}

\newcommand{\al}{\alpha}

\newcommand{\C}{{\mathbb C}}

\newcommand{\la}{\lambda}

\title{On the monoidality of Saito reflection functors}
\author{Syu \textsc{Kato} \footnote{Department of Mathematics, Kyoto University, Oiwake Kita-Shirakawa Sakyo Kyoto 606-8502 JAPAN \tt{E-mail:syuchan@math.kyoto-u.ac.jp}}}

\begin{document}
\maketitle

\begin{abstract}
We extend the definition of the Saito reflection functor of the Khovanov-Lauda-Rouquier algebras to symmetric Kac-Moody algebra case and prove that it defines a monoidal functor.
\end{abstract}

\section*{Introduction}

In \cite{Kat14}, the Saito reflection functors for the Khovanov-Lauda-Rouquier algebras of type $\mathsf{ADE}$ are introduced. It categorifies Lusztig's braid group action \cite[\S 39]{Lus93} on (a subalgebra of) of the positive half of the quantum groups in the sense of Khovanov-Lauda-Rouquier \cite{KL09, R}. They are main ingredients to construct PBW bases in the spirit of Lusztig \cite{Lus93}, and provided a certain role in the representation theory of the Khovanov-Lauda-Rouquier algebras.

The goal of this paper is to develop it little bit further, and provide some basic properties in more general setting than that of \cite{Kat14}. Let $\mathcal A := \mathbb Z [t ^{\pm 1}]$. Let $\mathfrak g$ be a symmetric Kac-Moody Lie algebra, and let $U ^+$ be the positive half of the $\mathcal A$-integral version of the quantum group of $\mathfrak g$ (see e.g. Lusztig \cite{Lus93} \S 1). Let $Q ^+ := \mathbb Z_{\ge 0} I$, where $I$ is the set of positive simple roots. We have a weight space decomposition $U^+ = \bigoplus _{\beta \in Q ^+} U^+ _{\beta}$. We have the Weyl group $W$ of $\mathfrak g$ with its set of simple reflections $\{ s_i \} _{i \in I}$. For each $\beta \in Q^+$, we have a finite set $B ( \infty )_{\beta}$ which parameterizes a pair of distinguished bases $\{ G ^{up} ( b ) \}_{b \in B ( \infty )_{\beta}}$ and $\{ G ^{low} ( b ) \} _{b \in B ( \infty )_{\beta}}$ of $\mathbb Q ( t ) \otimes _{\mathcal A} U^+ _{\beta}$. The Khovanov-Lauda-Rouquier algebra $R_{\beta}$ is a certain graded algebra whose grading is bounded from below with the following properties:
\begin{itemize}
\item The set of isomorphism classes of simple graded $R _{\beta}$-modules (up to grading shifts) is also parameterized by $B ( \infty )_{\beta}$;
\item For each $b \in B ( \infty )_{\beta}$, we have a simple graded $R_{\beta}$-module $L_b$ and its projective cover $P_b$. Let $L_{b'} \left< k \right>$ be the grade $k$ shift of $L_{b'}$, and let $[P_b : L_{b'} \left< k \right>]_0$ be the multiplicity of $L_{b'} \left< k \right>$ in $P_b$ (that is finite). Then, we have
$$G ^{low} ( b ) = \sum_{b' \in B ( \infty )_{\beta}, k \in \mathbb Z} t^k [P_b : L_{b'} \left< k \right>]_0 G^{up} ( b' );$$
\item For each $\beta, \beta' \in Q^+$, there exists an induction functor
$$\star : R_{\beta} \mathchar`-\mathsf{gmod} \times R_{\beta'} \mathchar`-\mathsf{gmod} \ni (M,N) \mapsto M \star N \in R_{\beta + \beta'} \mathchar`-\mathsf{gmod};$$
\item $\mathbf K := \bigoplus _{\beta \in Q^+} \mathbb Q ( t ) \otimes _{\mathcal A} K ( R_{\beta} \mathchar`-\mathsf{gmod} )$ is an associative algebra isomorphic to $\mathbb Q ( t ) \otimes _{\mathcal A} U^+$ with its product inherited from $\star$ (and the $t$-action is a grading shift).
\end{itemize}

For each $i \in I$ and $\beta \in Q^+$, we have certain quotients ${}_i R_{\beta}$ and ${}^i R_{\beta}$ of $R_{\beta}$. In case $s_i \beta \in Q^+$, an interpretation of Lusztig's geometric construction yields that ${}_i R_{\beta}$ and ${}^i R_{s_i \beta}$ must be Morita equivalent. This naturally enables us to define a right exact functor
$$\mathbb T _i : R _{\beta} \mathchar`-\mathsf{gmod} \longrightarrow \!\!\!\!\! \rightarrow {} ^i R _{\beta}\mathchar`-\mathsf{gmod} \stackrel{\cong}{\longrightarrow} {}_i R_{s_i \beta} \mathchar`-\mathsf{gmod} \hookrightarrow R_{s_{i} \beta} \mathchar`-\mathsf{gmod}$$
that we call the Saito reflection functor. Under this setting, our main results read:

\begin{fthm}[Theorems \ref{SRF} + \ref{braid} + \ref{main}]\label{fmain}
The functors $\{\mathbb T_i\}_{i \in I}$ satisfies the following:
\begin{enumerate}
\item There exist a right adjoint functor $\mathbb T^*_i$ of $\mathbb T_i$;
\item For each $M \in {}^i R _{\beta} \mathchar`-\mathsf{gmod}$ and $N \in {}_i R_{s_i \beta} \mathchar`-\mathsf{gmod}$, we have
$$\mathrm{ext} _{R_{s_{i} \beta}} ^{*} ( \mathbb T _i M, N ) \cong \mathrm{ext} _{R_{\beta}} ^{*} ( M, \mathbb T^* _i N );$$
\item They satisfy the braid relations;
\item For each $\beta_1, \beta_2 \in Q^+ \cap s_i Q^+$ and $M_1 \in {}^i R_{\beta_1} \mathchar`-\mathsf{gmod}$, $M_2 \in {}^i R_{\beta_2} \mathchar`-\mathsf{gmod}$, we have a natural isomorphism
$$\mathbb T_i ( M_1 \star M_2 ) \cong ( \mathbb T_i M_1 ) \star ( \mathbb T_i M_2 ).$$
Here we understand $M_1, M_2$ as modules of $R_{\beta_1}$ and $R_{\beta_2}$ through the pullbacks.
\end{enumerate}
\end{fthm}

We remark that Theorem \ref{fmain} confirms a conjecture in \cite{KKOP17} and provides one way to correct an error in \cite{Kat14} (see Remark \ref{ePBW} or the arXiv version of \cite{Kat14}). Also, the above result should extend to the positive characteristic case at least when $\mathfrak g$ is of type $\mathsf{ADE}$ by using \cite{M17}.

We note that Peter McNamara sent me a version of \cite{M17b} during the preparation of this paper that partly overlaps with the content of the paper.

\section{Conventions and recollections}

An algebra $R$ is a (not necessarily commutative) unital $\mathbb C$-algebra. A variety $\mathfrak X$ is a separated reduced scheme $\mathfrak X_0$ of finite type over some localization $\mathbb Z_S$ of $\mathbb Z$ specialized to $\mathbb C$. It is called a $G$-variety if we have an action of a connected affine algebraic group scheme $G$ flat over $\mathbb Z_S$ on $\mathfrak X_0$ (specialized to $\mathbb C$). As in \cite{BBD} \S 6 and \cite{BL94} (see alto \cite{Kat17}), we transplant the notion of weights to the derived category of ($G$-equivariant) constructible sheaves with finite monodromy on $\mathfrak X$. Let us denote by $D^b ( \mathfrak X )$ (resp. $D^+ ( \mathfrak X )$) the bounded (resp. bounded from the below) derived category of the category of constructible sheaves on $\mathfrak X$, and denote by $D^+ _G ( \mathfrak X )$ the $G$-equivariant derived category of $\mathfrak X$. We have a natural forgetful functor $D^+ _G ( \mathfrak X ) \to D^+ ( \mathfrak X )$, whose preimage of $D^b ( \mathfrak X )$ is denoted by $D ^b _G ( \mathfrak X )$. For an object of $D^b_G ( \mathfrak X )$, we may denote its image in $D ^b ( \mathfrak X )$ by the same letter.

\section{Quivers and the KLR algebras}\label{QKLR}
Let $\Gamma = (I, \Omega)$ be an oriented graph with the set of its vertex $I$ and the set of its oriented edges $\Omega$. Here $I$ is fixed, and $\Omega$ might change so that the underlying graph $\Gamma_0$ of $\Gamma$ is fixed. We have a symmetric Kac-Moody algebra $\mathfrak g$ with its Dynkin diagram $\Gamma_0$. We refer $\Omega$ as the orientation of $\Gamma$. We form a path algebra $\mathbb C [\Gamma]$ of $\Gamma$.

For $h \in \Omega$, we define $h' \in I$ to be the source of $h$ and $h'' \in I$ to be the sink of $h$. We denote $i \leftrightarrow j$ for $i, j \in I$ if and only if there exists $h \in \Omega$ such that $\{ h', h'' \} = \{ i, j \}$. A vertex $i \in I$ is called a sink of $\Gamma$ (or $\Omega$) if $h' \neq i$ for every $h \in \Omega$. A vertex $i \in I$ is called a source of $\Gamma$ (or $\Omega$) if $h'' \neq i$ for every $h \in \Omega$.

Let $Q^+$ be the free abelian semi-group generated by $\{ \alpha _i \} _{i \in I}$, and let $Q ^+ \subset Q$ be the free abelian group generated by $\{ \alpha _i \} _{i \in I}$. We sometimes identify $Q$ with the root lattice of $\mathfrak g$ with a set of its simple roots $\{ \alpha _i \} _{i \in I}$. Let $W = W ( \Gamma_0 )$ denote the Weyl group of type $\Gamma_0$ with a set of its simple reflections $\{ s_i \} _{i \in I}$. The group $W$ acts on $Q$ via the above identification. Let $R^+ := W \{ \alpha _i \} _{i \in I} \cap Q^+$ be the set of positive roots of $\mathfrak g$.

An $I$-graded vector space $V$ is a vector space over $\mathbb C$ equipped with a direct sum decomposition $V = \bigoplus _{i \in I} V_i$.

Let $V$ be an $I$-graded vector space. For $\beta \in Q ^+$, we declare $\underline{\dim} \, V = \beta$ if and only if $\beta = \sum _{i \in I} ( \dim V _i ) \alpha_i$. We call $\underline{\dim} \, V$ the dimension vector of $V$. Form a vector space
$$E_{V} ^{\Omega} := \bigoplus _{h \in \Omega} \mathrm{Hom} _{\mathbb C} ( V _{h'}, V _{h''} ).$$
We set $G _V := \prod_{i \in I} \mathop{GL} ( V _{i} )$. The group $G_V$ acts on $E_V ^{\Omega}$ through its natural action on $V$. The space $E_V ^{\Omega}$ can be identified with the based space of $\mathbb C [\Gamma]$-modules with its dimension vector $\beta$.

For each $k \ge 0$, we consider a sequence $\mathbf m = ( m_1,m_2,\ldots, m_{k} ) \in I ^{k}$. We abbreviate this as $\mathsf{ht} ( \mathbf m ) = k$. We set $\mathsf{wt} ( \mathbf m ) := \sum _{j = 1} ^k \alpha _{m_j} \in Q ^+$. For $\beta = \mathsf{wt} ( \mathbf m ) \in Q ^+$, we set $\mathsf{ht} \, \beta = k$. For a sequence $\mathbf m' := (m_1',\ldots, m'_{k'}) \in I ^{k'}$, we set
$$\mathbf m + \mathbf m' := ( m_1,\ldots, m_k, m_1',\ldots, m_{k'}') \in I ^{k+k'}.$$
For $i \in I$ and $k \ge 0$, we understand that $ki = (i,\ldots,i) \in I^k$.

For each $\beta \in Q ^+$, we set $Y ^{\beta}$ to be the set of all sequences $\mathbf m$ such that $\mathsf{wt} ( \mathbf m ) = \beta$. For each $\beta \in Q^+$ with $\mathsf{ht} \, \beta = n$ and $1 \le i < n$, we define an action of $\{\sigma_i\}_{i=1}^{n-1}$ on $Y^{\beta}$ as follows: For each $1 \le i < n$ and $\mathbf m = ( m_1, \ldots, m_{n} ) \in Y ^{\beta}$, we set
$$\sigma_i \mathbf m := ( m_1,\ldots, m_{i-1}, m_{i+1}, m_i, m_{i+2},\ldots, m_n).$$ 
It is clear that $\{\sigma_i\}_{i=1}^{n-1}$ generates a $\mathfrak S_n$-action on $Y^{\beta}$. In addition, $\mathfrak S_n$ naturally acts on a set of integers $\{ 1,2,\ldots, n \}$.

For $\mathbf m \in Y ^{\beta}$ and $1 \le i < \mathsf{ht} \, \beta$, we set $h_{\mathbf m, i} := \# \{ h \in \Omega \mid h' = m_i, h'' = m_{i+1}\}$ and $a_{\mathbf m, i} := h_{\mathbf m, i} + h_{\sigma_i \mathbf m, i}$.

\begin{defn}[Khovanov-Lauda \cite{KL09}, Rouquier \cite{R}]\label{KLR}
Let $\beta \in Q^+$ so that $n = \mathsf{ht} \, \beta$. We define the KLR algebra $R _{\beta}$ as a unital algebra generated by the elements $z _1, \ldots, z_n$, $\tau_1, \ldots, \tau_{n-1}$, and $e ( \mathbf m )$ $(\mathbf m \in Y ^{\beta})$ subject to the following relations:
\begin{enumerate}
\item $\deg z _i e ( \mathbf m ) = 2$ for every $i$, and
$$\deg \tau _i e ( \mathbf m ) = \begin{cases}-2 & (m _{i} = m_{i+1}) \\a_{\mathbf m, i} & (m_i \leftrightarrow m_{i+1}) \\ 0 & (otherwise) \end{cases};$$
\item $[ z _i, z _j ] = 0$, $e ( \mathbf m ) e ( \mathbf m' ) = \delta _{\mathbf m, \mathbf m'} e ( \mathbf m )$, and $\sum _{\mathbf m \in Y ^{\beta}} e ( \mathbf m ) = 1$;
\item $\tau _i e ( \mathbf m ) = e ( \sigma _i \mathbf m ) \tau _i e ( \mathbf m )$, and $\tau _i \tau _j e ( \mathbf m ) = \tau _j \tau _i e ( \mathbf m )$ for $|i-j|>1$;
\item $\tau _i ^2 e ( \mathbf m ) = Q _{\mathbf m, i} ( z _i, z _{i+1} ) e ( \mathbf m )$;
\item For each $1 \le i < n$, we have
\begin{align*}
 \tau_{i+1} \tau_i \tau_{i+1} e ( \mathbf m ) - & \tau_i \tau_{i+1} \tau_i e ( \mathbf m )\\
& = \begin{cases} \frac{Q _{\mathbf m, i} ( z _{i+2}, z _{i+1} ) - Q _{\mathbf m, i} ( z _i, z _{i+1} )}{z _{i+2} - z _i} e ( \mathbf m )& (m_{i+2} = m_i) \\ 0 & (\text{otherwise}) \end{cases};
\end{align*}
\item $\tau_i z _k e ( \mathbf m )- z _{\sigma_i (k)} \tau_i e ( \mathbf m ) = \begin{cases} - e ( \mathbf m ) & (i=k, m_i = m_{i+1}) \\ e ( \mathbf m ) & (i=k-1, m_i = m_{i+1}) \\ 0 & (\text{otherwise})\end{cases}$.
\end{enumerate}
Here we set
$$Q _{\mathbf m, i} ( u,v ) = \begin{cases} 1 & (m_i \neq m _{i+1}, m_i \not\leftrightarrow m_{i+1}) \\
(-1)^{h_{\mathbf m, i}} ( u - v )^{a_{\mathbf m, i}} & (m_i \leftrightarrow m_{i+1}) \\ 0 & (\text{otherwise}) \end{cases},$$
where $u,v$ are indeterminants. \hfill $\Box$
\end{defn}

\begin{rem}
Note that the algebra $R_{\beta}$ a priori depends on the orientation $\Omega$ through $Q _{\mathbf m, i} ( u,v )$. Since the graded algebras $R_{\beta}$ are known to be mutually isomorphic for any two choices of $\Omega$ $($cf. $\cite{R}$ \S 3.2.4 and Theorem \ref{VV}$)$, we suppress this dependence in the below.
\end{rem}

For an $I$-graded vector space $V$ with $\underline{\dim} \, V = \beta$, we define
\begin{align*}
F ^{\Omega} _{\beta} := & \Biggl\{ ( \{ F_j \} _{j=0} ^{\mathsf{ht} \beta }, x ) \Biggl| {\small \begin{matrix}  x \in E_{V} ^{\Omega}. \text{ For each $0 < j \le \mathsf{ht} \beta $,}\\
F _j \subset V \text{ is an $I$-graded vector subspace,}\\ F _{j+1} \subsetneq F_j \text{, and satisfies } x F _{j} \subset F _{j+1}. \end{matrix} }\Biggr\} \hskip 5mm \text{and}\\
\mathcal B ^{\Omega} _{\beta} := & \Biggl\{ \{ F_j \} _{j=0} ^{\mathsf{ht} \beta } \Biggl| {\small 
F _j \subset V \text{ is an $I$-graded vector subspace s.t. } F _{j+1} \subsetneq F_{j}. }\Biggr\}.
\end{align*}
We have a projection
$$\varpi _{\beta} ^{\Omega} : F ^{\Omega} _{\beta} \ni ( \{ F_j \} _{j=0} ^{\mathsf{ht} \beta }, x ) \mapsto \{ F_j \} _{j=0} ^{\mathsf{ht} \beta } \in \mathcal B ^{\Omega} _{\beta},$$
which is $G_V$-equivariant. For each $\mathbf m \in Y ^{\beta}$, we have a connected component
$$F ^{\Omega} _{\mathbf m} := \{( \{ F_j \} _{j=0} ^{\mathsf{ht} \beta }, x ) \in F ^{\Omega} _{\beta} \mid \underline{\dim} \, F_{j} / F_{j+1} = \alpha _{m_{j+1}} \hskip 2mm \forall j \} \subset F ^{\Omega} _{\beta},$$
that is smooth of dimension $d _{\mathbf m} ^{\Omega}$. We set $\mathcal B ^{\Omega} _{\mathbf m} := \varpi _{\beta} ^{\Omega} ( F ^{\Omega} _{\mathbf m} )$, that is an irreducible component of $\mathcal B ^{\Omega} _{\beta}$. Let
$$\pi ^{\Omega} _{\mathbf m} : F ^{\Omega} _{\mathbf m} \ni ( \{ F_j \} _{j=0} ^{\mathsf{ht} \beta }, x ) \mapsto x \in E ^{\Omega} _{V}$$
be the second projection that is also $G_V$-equivariant. The map $\pi ^{\Omega} _{\mathbf m}$ is projective, and hence
$$\mathcal L _{\mathbf m} ^{\Omega} := (\pi ^{\Omega} _{\mathbf m}) _! \, \underline{\mathbb C} \, [ d _{\mathbf m} ^{\Omega} ]$$
decomposes into a direct sum of (shifted) irreducible perverse sheaves with their coefficients in $D ^b ( \mathrm{pt} )$ (Gabber's decomposition theorem, \cite{BBD} \S 6.2.5). Let us denote by $\mathcal Q_{\mathbf m}^{\Omega}$ be the set of isomorphism classes of simple irreducible perverse sheaves that appear as a direct summand of $\mathcal L _{\mathbf m} ^{\Omega}$ (with some shifts). We set $\mathcal L ^{\Omega} _{\beta} := \bigoplus _{\mathbf m \in Y ^{\beta}} \mathcal L _{\mathbf m} ^{\Omega}$ and $\mathcal Q_{\beta}^{\Omega} := \bigcup _{\mathbf m \in Y ^{\beta}} \mathcal Q_{\mathbf m}^{\Omega}$. Let $e ( \mathbf m )$ be the idempotent in $\mathrm{End} ( \mathcal L ^{\Omega} _{\beta} )$ so that $e ( \mathbf m )  \mathcal L ^{\Omega} _{\beta} =  \mathcal L ^{\Omega} _{\mathbf m}$. Since $\pi ^{\Omega} _{\mathbf m}$ is projective, we conclude that $\mathbb D \mathcal L _{\mathbf m} ^{\Omega} \cong \mathcal L _{\mathbf m} ^{\Omega}$ for each $\mathbf m \in Y ^{\beta}$, and hence
\begin{equation}
\mathbb D \mathcal L ^{\Omega} _{\beta} \cong \mathcal L ^{\Omega} _{\beta}. \label{L-auto}
\end{equation}

\begin{thm}[Varagnolo-Vasserot \cite{VV11}]\label{VV}
Under the above settings, we have an isomorphism of graded algebras:
$$R _{\beta} \cong \bigoplus _{i \in \mathbb Z} \mathrm{Ext} ^i _{G_V} ( \mathcal L ^{\Omega} _{\beta}, \mathcal L ^{\Omega} _{\beta} ).$$
In particular, the RHS does not depend on the choice of an orientation $\Omega$ of $\Gamma_0$.
\end{thm}

For each $\mathbf m, \mathbf m' \in Y ^{\beta}$, we set
$$R _{\mathbf m, \mathbf m'} := e ( \mathbf m ) R _{\beta} e ( \mathbf m' ) =  \bigoplus _{i \in \mathbb Z} \mathrm{Ext} ^i _{G_V} ( \mathcal L ^{\Omega} _{\mathbf m'}, \mathcal L ^{\Omega} _{\mathbf m} ).$$

We set $S _{\beta} \subset R _{\beta}$ to be a subalgebra which is generated by $e ( \mathbf m )$ ($\mathbf m \in Y^{\beta}$) and $z_1,\ldots, z _n$.

For each $\beta_1, \beta _2 \in Q^+$ with $\mathsf{ht} \, \beta _1 = n_1$ and $\mathsf{ht} \, \beta _2 = n_2$, we have a natural inclusion:
$$
\xymatrix@=0pt{
& R _{\beta _1} \boxtimes R _{\beta _2} & \ni & e ( \mathbf m ) \boxtimes e ( \mathbf m' ) & \mapsto & e ( \mathbf m + \mathbf m' ) & \in & R _{\beta_1 + \beta _2}\\
& R_{\beta _1} \boxtimes 1 & \ni & z_i \boxtimes 1, \tau _i \boxtimes 1 & \mapsto & z _i, \tau_i & \in & R _{\beta_1 + \beta _2}\\
& 1 \boxtimes R _{\beta_2} & \ni & 1 \boxtimes z _i, 1 \boxtimes \tau _i & \mapsto & z _{i+n_1}, \tau _{i+n_1} & \in & R _{\beta_1 + \beta _2}
}.
$$
This defines an exact functor
$$\star : R_{\beta_1} \boxtimes R _{\beta _2} \mathchar`-\mathsf{gmod} \ni M _1 \boxtimes M _2 \mapsto R _{\beta _1 + \beta _2} \otimes _{R_{\beta_1} \boxtimes R _{\beta _2}} ( M _1 \boxtimes M _2 ) \in R _{\beta_1 + \beta_2} \mathchar`-\mathsf{gmod}.$$
The functor $\star$ restricts to an exact functor in the category of graded projective modules (see e.g. \cite{KL09} 2.16):
$$\star : R_{\beta_1} \boxtimes R _{\beta _2} \mathchar`-\mathsf{proj} \ni M _1 \boxtimes M _2 \mapsto R _{\beta _1 + \beta _2} \otimes _{R_{\beta_1} \boxtimes R _{\beta _2}} ( M _1 \boxtimes M _2 ) \in R _{\beta_1 + \beta_2} \mathchar`-\mathsf{proj}.$$

If $i \in I$ is a source of $\Gamma$ and $f = ( f_h ) _{h \in \Omega} \in E ^{\Omega}_V$, then we define
$$\epsilon^*_i ( f ) := \dim \ker \bigoplus_{h \in \Omega, h' = i} f _h \le \dim V _i.$$
If $i \in I$ is a sink of $\Gamma$ and $f = ( f_h ) _{h \in \Omega} \in E ^{\Omega}_V$, then we define
$$\epsilon_i ( f ) := \dim \mathrm{coker} \bigoplus_{h \in \Omega, h'' = i} f _h \le \dim V _i.$$
Each of $\epsilon^*_i ( f )$ or $\epsilon _i ( f )$ do not depend on the choice of a point in a $G_V$-orbit, and is a constructible function on $E_V^\Omega$. Hence, $\epsilon_i$ or $\epsilon^*_i$ induces a function on $E ^{\Omega}_V$ that is constant on each $G_V$-orbit, and a function on $\mathcal Q_{\beta}^{\Omega}$ through its value on an open dense subset of the support of its element whenever $i$ is a source or a sink.

\begin{prop}[Lusztig \cite{Lus91a}]\label{KScrys}
For each $i \in I$, the functions $\epsilon_i$ and $\epsilon_i ^*$ descend to functions on $\mathcal Q_{\beta}^{\Omega}$ for each $\beta \in Q^+$. In particular, it gives rise to functions on the set of isomorphism classes of simple graded $R_{\beta}$-modules $($up to degree shifts$)$.
\end{prop}

\begin{proof}
Note that \cite[Proposition 6.6]{Lus91a} considers only $\epsilon_i$, but $\epsilon_i ^*$ is obtained by swapping the order of the convolution operation.
\end{proof}

\begin{thm}[Khovanov-Lauda \cite{KL09}, Rouquier \cite{R}, Varagnolo-Vasserot \cite{VV11}]\label{VVR} In the above setting, we have:
\begin{enumerate}
\item For each $i \in I$ and $n \ge 0$, $R_{n \alpha _i}$ has a unique indecomposable projective module $P _{n i}$ up to grading shifts;
\item The functor $\star$ induces a $\mathbb Z [ t^{\pm 1} ]$-algebra structure on
$$\mathbf K := \bigoplus _{\beta \in Q^+} K ( R _{\beta} \mathchar`-\mathsf{proj} );$$
\item The algebra $\mathbf K$ is isomorphic to the integral form $U ^+$ of the positive part of the quantized enveloping algebra of type $\Gamma_0$ by identifying $[P_{ni}]$ with the $n$-th divided power of a Chevalley generator of $U ^+$;
\item The above isomorphism identifies the classes of indecomposable graded projective $R_{\beta}$-modules $(\beta \in Q ^+)$ with an element of the lower global basis of $U ^+$ in the sense of $\cite{Kas91}$;
\item There exists a set $B ( \infty ) = \bigsqcup _{\beta \in Q ^+} B ( \infty )_{\beta}$ that parameterizes indecomposable projective modules of $\bigoplus _{\beta \in Q ^+} R _{\beta} \mathchar`-\mathsf{gmod}$. This identifies the functions $\epsilon _i, \epsilon ^*_i$ $(i \in I)$ with the corresponding functions on $B ( \infty )$.
\end{enumerate}
\end{thm}

\begin{proof}
See \cite[Theorem 2.5]{Kat14}.
\end{proof}

\begin{prop}[\cite{Kat14}]\label{purity}
The sheaf $\mathcal L ^{\Omega} _{\beta}$ can be equipped with the structure of pure weight $0$. In particular, the graded algebra $R _{\beta}$ itself is pure of weight $0$.
\end{prop}

\begin{proof}
The statement of \cite[Proposition 2.7]{Kat14} is only when $\Gamma_0$ is a Dynkin quiver, but the argument works in general. 
\end{proof}

Thanks to Theorem \ref{VV} and Theorem \ref{VVR} 5), we have an identification $B ( \infty ) _{\beta} \cong \mathcal Q ^{\Omega} _\beta$. Via this idenfication, each $b \in B ( \infty ) _{\beta}$ defines a $G_V$-equivariant simple perverse sheaf $\mathsf{IC}^{\Omega} ( b )$ on $E^{\Omega}_V$, where $\underline{\dim} \, V = \beta$. Each $b \in B ( \infty ) _{\beta}$ defines an indecomposable graded projective module $P_b$ of $R_{\beta}$ with simple head $L_{b}$ that is isomorphic to its graded dual $L_b ^*$.

Let $\beta \in Q ^+$ so that $\mathsf{ht} \, \beta = n$. For each $i \in I$ and $k \ge 0$, we set
\begin{align*}
Y ^{\beta} _{k,i} & := \{ \mathbf m = (m_j) \in Y ^{\beta} \mid m _1 = \cdots = m _k = i \} \text{ and}\\
Y ^{\beta,*} _{k,i} & := \{ \mathbf m = (m_j) \in Y ^{\beta} \mid m _{n} = \cdots = m _{n-k+1} = i \}.
\end{align*}
In addition, we define two idempotents of $R_{\beta}$ as:
$$e_i ( k ) := \sum _{\mathbf m \in Y ^{\beta}_{k,i}} e ( \mathbf m ), \hskip 2mm \text{ and } \hskip 2mm e_i^* ( k ) := \sum _{\mathbf m \in Y ^{\beta,*}_{k,i}} e ( \mathbf m ).$$

\begin{thm}[Lusztig \cite{Lus90a} \S 6, Lauda-Vazirani \cite{LV11} \S 2.5.1]\label{crys}
Let $\beta \in Q^+$. For each $b \in B ( \infty ) _{\beta}$ and $i \in I$, we have
\begin{align*}
\epsilon_i ( b ) & = \max \{ k \!\mid e_i ( k ) L _b \neq \{ 0 \} \} \text{ and }\\
\epsilon_i^* ( b ) & = \max \{ k \!\mid e_i^* ( k ) L _b \neq \{ 0 \} \}.
\end{align*}
Moreover, $e_i ( \epsilon_i ( b ) ) L_b$ and $e^*_i ( \epsilon_i^* ( b ) ) L_b$ are irreducible $R_{\epsilon_i ( b ) \alpha _i} \boxtimes R _{\beta - \epsilon_i ( b ) \alpha _i}$-module and $R _{\beta - \epsilon_i^* ( b ) \alpha \i} \boxtimes R _{\epsilon_i^* ( b ) \alpha _i}$-module, respectively. In addition, if we have distinct $b,b' \in B ( \infty )_{\beta}$ so that $\epsilon _i ( b ) = k = \epsilon _i ( b' )$ with $k \ge 0$, then $e_{i} ( k ) L _b$ and $e_i ( k ) L_{b'}$ are not isomorphic as an $R_{k \alpha _i} \boxtimes R _{\beta - k \alpha _i}$-module. \hfill $\Box$
\end{thm}

\section{Saito reflection functors}\label{sSRF}
Let $\Omega _i$ be the set of edges $h \in \Omega$ with $h'' = i$ or $h' = i$. Let $s_i \Omega _i$ be a collection of edges obtained from $h \in \Omega _i$ by setting $(s_i h)' = h''$ and $(s_i h)'' = h'$. We define $s_i \Omega := ( \Omega \backslash \Omega _i ) \cup s _i \Omega _i$ and set $s_i \Gamma := ( I, s_i \Omega )$. Note that $\Gamma_0 = ( s_i \Gamma ) _0$.

Let $V$ be an $I$-graded vector space with $\underline{\dim} \, V = \beta$. For a sink $i$ of $\Gamma$, we define
$${}_i E _V ^{\Omega} := \bigr\{ ( f _h ) _{h \in \Omega} \in E _V ^{\Omega} \mid \mathrm{coker} ( \bigoplus _{h \in \Omega, h'' = i} f _h : \bigoplus _{h'} V _{h'} \to V _i ) =  \{ 0 \} \bigl\}.$$
For a source $i$ of $\Gamma$, we define
$${}^i E _V ^{\Omega} := \bigr\{ ( f _h ) _{h \in \Omega} \in E _V ^{\Omega} \mid \mathrm{ker} ( \bigoplus _{h \in \Omega, h' = i} f _h : V _i \to \bigoplus _{h''} V _{h''} ) =  \{ 0 \} \bigl\}.$$

Let $\Omega$ be an orientation of $\Gamma$ so that $i \in I$ is a sink. Let $\beta \in Q ^+ \cap s _i Q ^+$. Let $V$ and $V'$ be $I$-graded vector spaces with $\underline{\dim} \, V = \beta$ and $\underline{\dim} \, V' = s_i \beta$, respectively. We fix an isomorphism $\phi : \oplus _{j \neq i} V_j \stackrel{\cong}{\longrightarrow} \oplus _{j \neq i} V'_j$ as $I$-graded vector spaces. We define:
$$Z _{V,V'}^{\Omega} := \Biggl\{ \{ ( f _h ) _{h \in \Omega}, ( f' _h ) _{h \in s_i \Omega}, \psi \} \Biggl| {\small \begin{matrix} ( f _h ) \in {}_i E _V^{\Omega}, ( f'_h ) \in {}^i E _{V'} ^{s_i \Omega}, \\
\phi f_h = f'_{h} \phi \text{ for } h \not\in \Omega_i \\ \psi : V_i' \stackrel{\cong}{\longrightarrow} \mathrm{ker} ( \bigoplus _{h \in \Omega_i} f_h : \bigoplus _h V_{h'} \to V_i )\end{matrix}} \Biggr\}.$$

We have a diagram:
\begin{equation}
\xymatrix{E ^{\Omega} _V & {}_i E ^{\Omega} _V \ar@{_{(}->}[l] _{j_V}& Z _{V,V'}^{\Omega} \ar@{->>}[r]^{p^i_{V'}} \ar@{->>}[l] _{q^i _{V}} & {}^i E ^{s_i \Omega} _{V'} \ar@{^{(}->}[r] ^{\jmath _{V'}}& E ^{s_i \Omega} _{V'} \hskip 5mm.}\label{df}
\end{equation}
If we set
$$G_{V,V'} := \mathop{GL} ( V_i ) \times \mathop{GL} ( V'_i ) \times \prod _{j\neq i} \mathop{GL} ( V_j ) \cong \mathop{GL} ( V_i ) \times \mathop{GL} ( V'_i ) \times \prod _{j\neq i} \mathop{GL} ( V'_j ),$$
then the maps $p^i_{V'}$ and $q^i_{V}$ are $G_{V,V'}$-equivariant.

\begin{prop}[Lusztig \cite{Lus98}]\label{bd}
The morphisms $p ^i_V$ and $q ^i _{V}$ in $(\ref{df})$ are $\mathrm{Aut} ( V _i )$-torsor and $\mathrm{Aut} ( V' _i )$-torsor, respectively. \hfill $\Box$
\end{prop}

When $\beta = \underline{\dim} \, V$, we set
$${} _i R _{\beta} ^{\Omega} := \mathrm{Ext} ^{\bullet} _{G_V} ( j_V ^* \mathcal L ^{\Omega} _{V}, j_V ^* \mathcal L ^{\Omega} _{V} ) \hskip 2mm \text{ and } \hskip 2mm {} ^i R _{s_i \beta} ^{s_i \Omega} := \mathrm{Ext} ^{\bullet} _{G_{V'}} ( \jmath _{V'}^* \mathcal L ^{s_i \Omega} _{V'}, \jmath _{V'} ^* \mathcal L ^{s_i \Omega} _{V'} ).$$

For each $k > 0$, we fix an $I$-graded vector subspace $U _{k} \subset V$ so that $\underline{\dim} \, U_{k} = \beta - k \alpha_{i}$ and an $I$-graded vector subspace $U' _{k} \subset V'$ so that $\underline{\dim} \, U'_{k} = s_{i} \beta - k \alpha_{i}$. We have natural embeddings $\kappa _{k}: E_{U_{k}}^{\Omega} \subset E_{V}^{\Omega}$ and $\eta _{k}: E_{U'_{k}}^{s_{i} \Omega} \subset E_{V'}^{s_{i} \Omega}$ by adding a trivial $\C [\Gamma]$-module of its dimension vector $k \al_i$.

\begin{thm}[Lusztig \cite{Lus91a}]\label{Res}
Let $k > 0$. The restriction $\kappa _{k} ^{*} \mathcal L_{\beta}^{\Omega}$ is a direct sum of shifted perverse sheaves in $\mathcal Q_{\beta - k \alpha_{i}} ^{\Omega}$. Similarly, the restriction $\eta _{k} ^{*} \mathcal L_{s_{i}\beta}^{s_{i} \Omega}$ is a direct sum of shifted perverse sheaves in $\mathcal Q_{s_i \beta - k \alpha_{i}} ^{s_{i} \Omega}$.
\end{thm}

\begin{proof}
The assertion is exactly \cite{Lus91a} Proposition 4.2 since the projection map $p$ (in the notation of \cite{Lus91a}) is an isomorphism if we appropriately arrange $\mathbf W$ and $\mathbf T$ in \cite[\S 4.1]{Lus91a}.
\end{proof}

We set ${}_{i} E_{V,k}^{\Omega} := G_V E_{U_k}^{\Omega}$, ${}_{i} E_{V,(k)}^{\Omega} := {}_{i} E_{V,k}^{\Omega} \backslash {}_{i} E_{V,k+1}^{\Omega}$, $i_k : {}_{i} E_{V,k}^{\Omega} \hookrightarrow E_V^{\Omega}$ and $j_k : {}_{i} E_{V,(k)}^{\Omega} \hookrightarrow {}_{i} E_{V,k}^{\Omega}$ for each $k > 0$. We have ${}_i E_{V,k}^{\Omega} = \bigsqcup_{k' \ge k} {}_{i} E_{V,(k')}^{\Omega}$, and we have $\epsilon_i ( x ) = k$ for $x \in {}_{i} E_{V,(k)}^{\Omega}$. The map $i_k$ is closed immersion, and the map $j_k$ is an open embedding. We set ${}^{i} E_{V',k}^{s_i\Omega} := G_{V'} E_{U'_k}^{s_i\Omega}$, and we define similar maps $\imath_k, \jmath_k$ for them that we use only as ``an analogous" situation.

\begin{prop}\label{i-Res}
Let $k > 0$. The sheaf $i_k ^{*} \mathcal L_{\beta}^{\Omega}$ is the direct sum of shifted perverse sheaves in $\mathcal Q_{\beta} ^{\Omega}$ supported on ${}_{i} E_{V,k}^{\Omega}$ if we restrict them to ${}_{i} E_{V,(k)}^{\Omega}$. Similarly, the restriction $\imath_{k} ^{*} \mathcal L_{s_{i}\beta}^{s_{i} \Omega}$ is a direct sum of shifted perverse sheaves in $\mathcal Q_{s_i \beta} ^{s_{i} \Omega}$ supported on ${}^{i} E_{V',k}^{s_i\Omega}$ along the loci with $\epsilon_i^* = k$.
\end{prop}

\begin{proof}
As the proofs of the both cases are completely parallel, we concentrate to the case of $i_k ^{*} \mathcal L_{\beta}^{\Omega}$. 

The map $\kappa_k$ factors through $i_k$ as
$$E_{U_k}^{\Omega} \stackrel{\kappa_k'}{\longrightarrow} {}_i E_{V,k}^{\Omega} \stackrel{i_k}{\longrightarrow} E_{V}^{\Omega}$$
for each $k$. Thus, Theorem \ref{Res} asserts that $( \kappa _k' )^* i_k ^{*} \mathcal L_{\beta}^{\Omega}$ is a direct sum of shifted perverse sheaves in $\mathcal Q_{\beta - k \alpha_{i}} ^{\Omega}$. We set $n := \dim \, V_i$. Let $P_k \subset \mathop{GL} ( n ) \cong \mathop{GL} ( V _i )$ be the parabolic subgroup so that its Levi part is $\mathop{GL} ( n - k ) \times \mathop{GL} ( k )$ and stabilizes $E_{U_k}^{\Omega} \subset E_{V}^{\Omega}$. Then, we have a map
$$\pi_k : \mathop{GL} ( n ) \times_{P_k} E_{U_k}^{\Omega} \longrightarrow E_V^{\Omega},$$
that is projective over the image. Note that $\pi_k$ is locally trivial fibration over ${}_{i} E_{V,(k)}^{\Omega}$ with its fiber isomorphic to $\mathrm{Gr} ( k, n )$. 

The sheaf $(\pi_k )_* \pi_k^* i_k ^{*} \mathcal L_{\beta}^{\Omega}$ can be regarded as the induction of the sheaf $\kappa^*_k \mathcal L_{\beta}^{\Omega}$, and hence it is a direct sum of shifted perverse sheaves in $\mathcal Q_{\beta} ^{\Omega}$. The above argument tells us that $i_k ^{*} \mathcal L_{\beta}^{\Omega}$ is a direct summand of $(\pi_k )_* \pi_k^* i_k ^{*} \mathcal L_{\beta}^{\Omega}$ when restricted to ${}_{i} E_{V,(k)}^{\Omega}$. Therefore, we conclude that $i_k ^{*} \mathcal L_{\beta}^{\Omega}$ is a direct sum of shifted perverse sheaves in $\mathcal Q_{\beta} ^{\Omega}$ supported on ${}_{i} E_{V,k}^{\Omega}$ restricted to ${}_{i} E_{V,(k)}^{\Omega}$ as required.
\end{proof}

For each $k > 0$, we define
$${} _i R _{\beta, k} ^{\Omega} := \mathrm{Ext} ^{\bullet} _{G_V} ( j_{V,k} ^* \mathcal L ^{\Omega} _{V}, j_{V,k} ^* \mathcal L ^{\Omega} _{V} ),$$
where $j_{V, k} : E_{V}^{\Omega} \backslash {}_{i} E_{V,k}^{\Omega} \hookrightarrow E_{V}^{\Omega}$. By definition, we have ${} _i R _{\beta, 1} ^{\Omega} = {} _i R _{\beta} ^{\Omega}$. By convention, we have $j_{V,k} = \mathrm{id}$ for $k > \dim \, V_i$, and we have ${} _i R _{\beta, k} ^{\Omega} = R _{\beta} ^{\Omega}$ in this case. We also define ${} ^i R _{s_i \beta, k} ^{s_i \Omega}$ in a similar fashion, that we use only as ``an analogous" situation.

\begin{thm}\label{i-quotients}
For each $k > 0$, we have an algebra isomorphism
$${} _i R _{\beta, k} ^{\Omega} \cong R_{\beta} / ( R_{\beta} e_i ( k ) R_{\beta} ).$$
Moreover, ${} _i R _{\beta, k+1} ^{\Omega} e_i ( k ) {} _i R _{\beta, k+1} ^{\Omega}$ is projective as a ${} _i R _{\beta, k+1} ^{\Omega} $-module. Similarly, the algebra ${} ^i R _{s_i \beta, k} ^{s_i \Omega}$ is isomorphic to $R _{s_i \beta} ^{s_i \Omega} / ( R _{s_i \beta} ^{s_i \Omega} e_i ^* ( k ) R _{s_i \beta} ^{s_i \Omega} )$, and ${} ^i R _{s_i \beta, k+1} ^{s_i \Omega} e_i^{*} ( k ) {} ^i R _{s_i \beta, k+1} ^{s_i \Omega}$ is projective as a ${} ^i R _{s_i \beta, k+1} ^{s_i \Omega}$-module. In particular, the algebras ${} _i R _{\beta, k} ^{\Omega}$ and ${} ^i R _{s_i \beta, k} ^{s_i \Omega}$ do not depend on the choice of $\Omega$.
\end{thm}

\begin{proof}
Since the case of ${} ^i R _{s_i \beta, k} ^{s_i \Omega}$ is completely parallel, we concentrate to the case of ${} _i R _{\beta, k} ^{\Omega}$. The case $k \gg 0$ is clear, and hence we prove the assertion by the downward induction on $k$. In particular, we assume that 
$${} _i R _{\beta, k+1} ^{\Omega} \cong R_{\beta} / ( R_{\beta} e_i ( k+1 ) R_{\beta} )$$
to prove our assertion. We denote ${} _i R _{\beta, k+1} ^{\Omega}$ by $R_{\beta, k+1}$ for simplicity.

We have
\begin{align*}
\mathrm{Ext} ^{\bullet} _{G_V} ( j_{V, k} ^* \mathcal L ^{\Omega} _{V}, j_{V, k} ^* \mathcal L ^{\Omega} _{V} ) & \cong \mathrm{Ext} ^{\bullet} _{G_V} ( j_{V, k} ^! \mathcal L ^{\Omega} _{V}, j_{V, k} ^! \mathcal L ^{\Omega} _{V} )\\
& \cong \mathrm{Ext} ^{\bullet} _{G_V} ( ( j_{V, k} ) _! j_{V, k} ^! \mathcal L ^{\Omega} _{V}, \mathcal L ^{\Omega} _{V} ).
\end{align*}
We set $E_k := ( E_{V}^{\Omega} \backslash {}_i E_{V, k+1} ^{\Omega} )$. By assumption, we can restrict ourselves to $E_k$ to compute the $\mathrm{Ext}$-groups. Hence, we freely assume that our maps are restricted to $E_k$ unless otherwise stated.

We have a distinguished triangle
\begin{equation}
( j_{V, k} ) _! j_{V, k} ^! \mathcal L ^{\Omega} _{V} \rightarrow \mathcal L ^{\Omega} _{V} \to ( i_{V,k} )_{*} i_{V,k}^{*} \mathcal L ^{\Omega} _{V} \stackrel{+1}{\longrightarrow},\label{distL}
\end{equation}
where $i_{V,k} : {}_i E_{V, (k)}^{\Omega} \hookrightarrow E_k$ is the complement inclusion. This yields an exact seqeunce
\begin{align*}
\mathrm{Ext} ^{\bullet} _{G_V} ( ( i_{V,k} )_{*} i_{V,k}^{*} \mathcal L ^{\Omega} _{V}, \mathcal L ^{\Omega} _{V} )  \rightarrow \mathrm{Ext} ^{\bullet} _{G_V} ( \mathcal L ^{\Omega} _{V}, \mathcal L ^{\Omega} _{V} ) & \stackrel{\psi}{\longrightarrow}\mathrm{Ext} ^{\bullet} _{G_V} ( ( j_{V,k} ) _! j_{V,k} ^! \mathcal L ^{\Omega} _{V}, \mathcal L ^{\Omega} _{V} ) \\
& \rightarrow \mathrm{Ext} ^{\bullet+1} _{G_V} ( ( i_{V,k} )_{*} i_{V,k}^{*} \mathcal L ^{\Omega} _{V}, \mathcal L ^{\Omega} _{V} )
\end{align*}
as $R_{\beta, k+1}$-modules. Note that $L_{b}$ is the coefficient of $\mathsf{IC} ^{\Omega} ( b )$ in $\mathcal L^{\Omega}_{\beta}$, and hence its support is contained in ${}_i E_{V,k}^{\Omega}$ when when $\epsilon_i ( k ) L_b \neq \{ 0 \}$. In particular, the simple $R_{\beta}^{\Omega}$-module $L_{b}$ contributes to $\mathrm{Ext} ^{\bullet} _{G_V} ( j_{V,k} ^! \mathcal L ^{\Omega} _V, j_{V,k} ^! \mathcal L ^{\Omega} _{V} )$ by (graded) Jordan-H\"older multiplicity zero when $\epsilon_i ( k ) L_b \neq \{ 0 \}$. It follows that $R_{\beta, k+1} e_i ( k ) R_{\beta, k+1} \subset \ker \, \psi$.

The action of $H^{\bullet}_{G_V} ( \mathrm{pt} )$ on $\mathrm{Ext} ^{\bullet} _{G_V} ( \mathcal L ^{\Omega} _{V}, \mathcal L ^{\Omega} _{V} )$ is through the center of $R_{\beta}$ (see e.g. \cite{VV11}), and it is torsion-free. Hence, the action of $H^{\bullet}_{\mathop{GL} ( V_i )} ( \mathrm{pt} )$ and $H^{\bullet}_{G_V} ( \mathrm{pt} )$ on $R_{\beta, k+1} \cong \mathrm{Ext} ^{\bullet} _{G_V} ( j_{V, k+1} ^! \mathcal L ^{\Omega} _{V}, j_{V, k+1} ^! \mathcal L ^{\Omega} _{V} )$ factors through the center of $R _{\beta, k+1}$.

Since $( i_{V, k} )_{*} = ( i_{V, k} )_{!}$, we have
$$\mathrm{Ext} ^{\bullet} _{G_V} ( ( i_{V, k} )_{*} i_{V, k}^{*} \mathcal L ^{\Omega} _{V}, \mathcal L ^{\Omega} _{V} )\cong \mathrm{Ext} ^{\bullet} _{G_V} ( i_{V, k}^{*} \mathcal L ^{\Omega} _{V}, i_{V, k} ^{!} \mathcal L ^{\Omega} _{V} ).$$
By our convention, $i_{V, k}^{*} \mathcal L ^{\Omega} _{V}$ and $i_{V, k}^{!} \mathcal L ^{\Omega} _{V}$ are supported on ${}_i E_{V,(k)}^{\Omega}$. In addition, we have ${}_i E_{V,(k)}^{\Omega} \cong \mathop{GL} ( V_i ) \times _{P_k} ( {}_i E_{U_k}^{\Omega} \cap {}_i E_{V,(k)}^{\Omega} )$ for a parabolic subgroup $P_k \subset \mathop{GL} ( V_i )$ borrowed from the proof of Proposition \ref{i-Res}. Here, the subgroup $\mathop{GL} ( k ) \subset P_k$ acts on ${}_i E_{U_k}^{\Omega}$ trivially. From this and the induction equivalence (\cite[\S 2.6.3]{BL94}), we obtain a free action of $H^{\bullet}_{\mathop{GL} ( k )} ( \mathrm{pt} )$ on $\mathrm{Ext} ^{\bullet} _{\mathop{GL} ( V_i )} ( ( i_{V, k} )_{*} i_{V, k}^{*} \mathcal L ^{\Omega} _{V}, \mathcal L ^{\Omega} _{V} )$. The image of the pullback map $H^{\bullet}_{\mathop{GL} ( V_i )}  ( \mathrm{pt} ) \rightarrow H^{\bullet}_{\mathop{GL} ( k )} ( \mathrm{pt} )$ contain $k$-algebraically independent elements (over the base field $\C$). From these, we conclude that the $H^{\bullet}_{\mathop{GL} ( V_i )}  ( \mathrm{pt} )$-action on $\mathrm{Ext} ^{\bullet} _{\mathop{GL} ( V_i )} ( ( i_{V, k} )_{*} i_{V, k}^{*} \mathcal L ^{\Omega} _{V}, \mathcal L ^{\Omega} _{V} )$ contains at least $k$ algebraically independent elements that acts torsion-freely.

On the other hand, the action of $H^{\bullet}_{\mathop{GL} ( V_i )} ( \mathrm{pt} )$ on $\mathrm{Ext} ^{\bullet} _{\mathop{GL} ( V_i )} ( j_{V, k} ^!  \mathcal L ^{\Omega} _{V}, j_{V, k} ^! \mathcal L ^{\Omega} _{V} )$ arises from the $\mathop{GL} ( V_i )$-action on some algebraic stratification of $E_{k-1}$ (see e.g. Chriss-Ginzburg \cite[3.2.23 and 8.4.8]{CG97}) so that the stalks of elements of $\mathcal Q_{\beta}^{\Omega}$ are constant (by the construction of $\mathcal Q_{\beta}^{\Omega}$; note that our stratification is finite). In other words, we have a finite $G_V$-stable stratification
$$E_V^{\Omega} \backslash {}_i E_{V,k} ^{\Omega} = \bigsqcup_{\la \in \Lambda} S_{\la}$$
and a complex of locally constant sheaves $\mathcal E_{\la}$ (obtained by a successive application of recollements) over $S_{\la}$ so that $\mathrm{Ext} ^{\bullet} _{\mathop{GL} ( V_i )} ( j_{V, k} ^!  \mathcal L ^{\Omega} _{V}, j_{V, k} ^! \mathcal L ^{\Omega} _{V} )$ is written as a finite successive distinguished triangles using $H^{\bullet}_{\mathop{GL} ( V_i )} ( S_{\la}, \mathcal E_{\la} )$. Moreover, $\mathrm{Ext} ^{\bullet} _{\mathop{GL} ( V_i )} ( j_{V, k} ^!  \mathcal L ^{\Omega} _{V}, j_{V, k} ^! \mathcal L ^{\Omega} _{V} )$ must be a finitely generated $H^{\bullet}_{\mathop{GL} ( V_i )} ( \mathrm{pt} )$-module as a result of the the fact that $\mathcal L ^{\Omega} _{V}$ is a finite direct sum of constructible complexes over $E_V^{\Omega}$.

The rank of the stabilizer of the $\mathop{GL} ( V_i )$-action on a point of $E_{k-1}$ is always $< k$. As a consequence, the action of $H^{\bullet}_{\mathop{GL} ( V_i )} ( \mathrm{pt} )$ on $H^{\bullet}_{\mathop{GL} ( V_i )} ( S_{\la}, \mathcal E_{\la} )$ (for every $\la \in \Lambda$) cannot carry $k$-algebraically independent elements that act torsion-freely. Therefore, the same holds for $\mathrm{Ext} ^{\bullet} _{\mathop{GL} ( V_i )} ( j_{V,k} ^!  \mathcal L ^{\Omega} _{V}, j_{V,k} ^! \mathcal L ^{\Omega} _{V} )$. Thus, the map
$$\mathrm{Ext} ^{\bullet} _{\mathop{GL} ( V_i )} ( ( j_{V,k} ) _! j_{V,k}^! \mathcal L ^{\Omega} _{V}, \mathcal L ^{\Omega} _{V} ) \rightarrow \mathrm{Ext} ^{\bullet+1} _{\mathop{GL} ( V_i )} ( ( i_{V,k} )_{*} i_{V,k}^{*} \mathcal L ^{\Omega} _{V}, \mathcal L ^{\Omega} _{V} )$$
must be nullity as we do not have enough number of algebraically independent elements of $H^{\bullet}_{\mathop{GL} ( V_i )} ( \mathrm{pt} )$ that acts on the LHS without torsion. By imposing the $G_V$-equivariance, we obtain a map
\begin{align*}
H^{\bullet}_{G_V / \mathop{GL} ( V_i )}( \mathrm{pt} ) \otimes \mathrm{Ext} ^{\bullet} _{\mathop{GL} ( V_i )} & ( ( j_{V,k} ) _! j_{V,k}^! \mathcal L ^{\Omega} _{V}, \mathcal L ^{\Omega} _{V} ) \\
 \rightarrow & H^{\bullet}_{G_V / \mathop{GL} ( V_i )}( \mathrm{pt} ) \otimes \mathrm{Ext} ^{\bullet+1} _{\mathop{GL} ( V_i )} ( ( i_{V,k} )_{*} i_{V,k}^{*} \mathcal L ^{\Omega} _{V}, \mathcal L ^{\Omega} _{V} )
\end{align*}
that induces a map
$$\mathrm{Ext} ^{\bullet} _{G_V} ( ( j_{V,k} ) _! j_{V,k}^! \mathcal L ^{\Omega} _{V}, \mathcal L ^{\Omega} _{V} ) \rightarrow \mathrm{Ext} ^{\bullet+1} _{G_V} ( ( i_{V,k} )_{*} i_{V,k}^{*} \mathcal L ^{\Omega} _{V}, \mathcal L ^{\Omega} _{V} )$$
through the spectral sequence (pulling back to the classifying space of $G_V$). This map must be also nullity as it is induced from the nullity.

Hence, we conclude a short exact seqeuence
$$0 \rightarrow \mathrm{Ext} ^{\bullet} _{G_V} ( ( i_{V,k} )_{*} i_{V,k}^{*} \mathcal L ^{\Omega} _{V}, \mathcal L ^{\Omega} _{V} ) \rightarrow  R_{\beta, k+1} \stackrel{\psi}{\longrightarrow} R_{\beta, k} \to 0$$
as left $R_{\beta, k+1}$-modules.

By Proposition \ref{i-Res}, the sheaf $( i_{V, k} )_{*} i_{V, k}^{*} \mathcal L ^{\Omega} _{V}$ is a direct sum of shifted perverse sheaves on $E_{k}$, that is supported on ${}_{i} E_{V, k}^{\Omega}$ (or ${}_{i} E_{V, (k)}^{\Omega}$). It follows that the graded $R_{\beta, k+1}$-module $\mathrm{Ext} ^{\bullet} _{G_V} ( ( i_{V, k} )_{*} i_{V, k}^{*} \mathcal L ^{\Omega} _{V}, \mathcal L ^{\Omega} _{V} )$ is the direct sum of projective covers of $L_{b}$ with $\epsilon_{i} ( b ) = k$. Since $R_{\beta, k+1} e_i ( k ) R_{\beta, k+1}$ is the maximal left $R_{\beta}$-submodule of $R_{\beta, k+1}$ generated by irreducible constituents $\{L_{b}\}_{\epsilon_{i} ( b ) = k}$, we deduce
$$R_{\beta, k+1} e_i ( k ) R_{\beta, k+1} \cong \mathrm{Ext} ^{\bullet} _{G_V} ( ( i_{V,k} )_{*} i_{V,k}^{*} \mathcal L ^{\Omega} _{V}, \mathcal L ^{\Omega} _{V} ) \cong \mathrm{Ext} ^{\bullet} _{G_V} (  \mathcal L ^{\Omega} _{V},( i_{V,k} )_{*} i_{V,k}^{!} \mathcal L ^{\Omega} _{V} ),$$
where the latter modules are actually calculated on $E_k$. Therefore, we conclude the assertions for $R_{\beta, k}$ as required.

This proceeds the induction step, and we conclude the assertion.
\end{proof}

\begin{cor}
The set of isomorphism classes of graded simple modules of ${} _i R _{\beta} ^{\Omega}$ and ${} ^i R _{s_i \beta} ^{s_i \Omega}$ are $\{ L_b \left< j \right>\} _{\epsilon _i ( b ) = 0, j \in \mathbb Z}$ and $\{ L_b \left< j \right>\} _{\epsilon ^* _i ( b ) = 0, j \in \mathbb Z}$, respectively. \hfill $\Box$
\end{cor}

\begin{thm}[\cite{Lus98}]\label{Tcorr}
The maps $q^i_V$ and $p^i_{V'}$ give rise to a bijective correspondence between perverse sheaves corresponding to $\{ b \in B ( \infty ) _{\beta} \mid \epsilon_{i} ( b ) = 0 \}$ and $\{ b \in B ( \infty ) _{s_{i}\beta} \mid \epsilon^{*}_{i} ( b ) = 0 \}$.
\end{thm}

\begin{proof}
In view of Theorem \ref{i-quotients}, the combination of \cite[Theorem 8.6]{Lus98} and Proposition \ref{KScrys} implies the result (see also \cite[Proposition 38.1.6]{Lus93}).
\end{proof}

\begin{prop}[\cite{Kat14}]\label{TM}
In the setting of Proposition \ref{bd}, two graded algebras ${} _i R _{\beta} ^{\Omega}$ and ${} ^i R _{s_i \beta} ^{s_i \Omega}$ are Morita equivalent to each other. In addition, this Morita equivalence is independent of the choice of $\Omega$ $($as long as $i$ is a sink$)$.
\end{prop}

\begin{proof}
Although the original setting in \cite[Proposition 3.5]{Kat14} is only for types $\mathsf{ADE}$, the arguments carry over to this case in view of Theorem \ref{Tcorr}.
\end{proof}

For each $b \in B ( \infty )_{s_i \beta}$, we denote by $T_i ( b ) \in B ( \infty ) _{\beta} \sqcup \{ \emptyset \}$ the element so that
$$( p ^i _{V'} )^{*} \mathsf{IC} ^{s_{i} \Omega} ( b ) [( \dim \, V_{i} )^{2}] \cong ( q^i _V )^{*} \mathsf{IC} ^{\Omega} ( T_i ( b ) ) [( \dim \, V'_{i} )^{2}],$$
(we understand that $T _i ( b ) = \emptyset$ if $\mathrm{supp} \, \mathsf{IC} ^{s_{i} \Omega} ( b ) \not\subset \mathrm{Im} \, p ^i _{V'}$). Note that $T_i ( b ) = \emptyset$ if and only if $\epsilon ^* _i ( b ) > 0$. In addition, we have $\epsilon _i ( T_i ( b ) ) = 0$ if $T_i ( b ) \neq \emptyset$. We set $T_i ^{-1} (b') := b$ if $b' = T_i ( b ) \neq \emptyset$.

Thanks to Theorem \ref{i-quotients}, we can drop $\Omega$ or $s_i \Omega$ from ${} _i R _{\beta} ^{\Omega}$ and ${} ^i R _{\beta} ^{s _i \Omega}$. We define a left exact functor
$$\mathbb T^* _i : R _{\beta} \mathchar`-\mathsf{gmod} \longrightarrow \!\!\!\!\! \rightarrow {} _i R _{\beta}\mathchar`-\mathsf{gmod} \stackrel{\cong}{\longrightarrow} {}^i R_{s_i \beta} \mathchar`-\mathsf{gmod} \hookrightarrow R_{s_{i} \beta} \mathchar`-\mathsf{gmod},$$
where the first functor is $\mathrm{Hom} _{R_{\beta}} ( {} _i R_{\beta}, \bullet )$, the second functor is Proposition \ref{TM}, and the third functor is the pullback. Similarly, we define a right exact functor
$$\mathbb T _i : R _{\beta} \mathchar`-\mathsf{gmod} \longrightarrow \!\!\!\!\! \rightarrow {} ^i R _{\beta}\mathchar`-\mathsf{gmod} \stackrel{\cong}{\longrightarrow} {}_i R_{s_i \beta} \mathchar`-\mathsf{gmod} \hookrightarrow R_{s_{i} \beta} \mathchar`-\mathsf{gmod},$$
where the first functor is ${}^i R_{\beta} \otimes_{R_{\beta}} \bullet$. We call these functors the Saito reflection functors (\cite[\S 3]{Kat14}). By the latter part of Proposition \ref{TM}, we see that these functors are independent of the choices involved.

\begin{thm}[\cite{Kat14} Theorem 3.6]\label{SRF}
Let $i \in I$. We have:
\begin{enumerate}
\item For each $b \in B ( \infty )_{\beta}$, we have
$$\mathbb T_i L _b = \begin{cases} L_{T_i(b)} & (\epsilon^*_i (b) = 0)\\ \{0\} & (\epsilon^*_i (b) > 0)\end{cases}, \text{ and } \hskip 3mm \mathbb T_i^* L _b = \begin{cases} L_{T_i^{-1}(b)} & (\epsilon_i (b) = 0)\\ \{0\} & (\epsilon_i (b) > 0)\end{cases};$$
\item The functors $( \mathbb T_i, \mathbb T^*_i)$ form an adjoint pair;
\item For each $M \in {}^i R _{\beta} \mathchar`-\mathsf{gmod}$ and $N \in {}_i R_{s_i \beta} \mathchar`-\mathsf{gmod}$, we have
$$\mathrm{ext} _{R_{s_{i} \beta}} ^{*} ( \mathbb T _i M, N ) \cong \mathrm{ext} _{R_{\beta}} ^{*} ( M, \mathbb T^* _i N ).$$
\end{enumerate}
\end{thm}

\begin{proof}
The proof of the first assertion is the same as \cite[Theorem 3.6]{Kat14} if we replace standard modules with projective modules, that involves only simple perverse sheaves. The proof of the second assertion is exactly the same as \cite[Theorem 3.6]{Kat14}. The third assertion requires the second part of Theorem \ref{i-quotients} instead of \cite[Corollary 1.6]{Kat14} (and also we need to repeat projective resolutions inductively on $\epsilon_i$ and $\epsilon_i^*$ in a downward fashion).
\end{proof}

\begin{thm}\label{braid}
Let $i, j \in I$. We have:
\begin{itemize}
\item If $i \not\leftrightarrow j$, then we have $\mathbb T_i \mathbb T_j \cong \mathbb T_j \mathbb T_i$;
\item If $\# \{ h \in \Omega \mid \{ h', h''\} = \{i,j\}\} = 1$, then we have $\mathbb T_i \mathbb T_j \mathbb T_i \cong \mathbb T_j \mathbb T_i \mathbb T_j$.
\end{itemize}
The same is true for $\mathbb T_i^*$ and $\mathbb T_j^*$. 
\end{thm}

\begin{proof}
By \cite[\S 9.4]{Lus98}, the functor $\mathbb T_i$ induces an isomorphism described in \cite[Lemma 38.1.3]{Lus93} (see also \cite{XZ17}). Hence, \cite[Theorem 39.4.3]{Lus93} (cf. Theorem \ref{SRF} 1)) implies that the both sides give the same correspondence between simple modules. As each of $\mathbb T_i$ transplants the simple modules and annihilates all the submodule that contains some specific simple modules (that induces an equivalence between some Serre subcategories), the same is true for their composition. Therefore, we conclude the result.
\end{proof}

\section{Monoidality of the Saito reflection functor}

We work in the setting of \S \ref{QKLR}. The goal of this section is to prove the following:

\begin{thm}\label{main}
Let $i \in I$, and let $\beta_1, \beta _2 \in Q^+$ so that $s_i \beta_1, s_i \beta_2 \in Q^+$. There exists a natural transformation
$$\mathbb T_i^* ( \bullet \star \bullet ) \longrightarrow \mathbb T_i^* ( \bullet ) \star \mathbb T_i^* ( \bullet )$$
as functors from the category of ${}_i R_{\beta_1} \boxtimes {}_i R_{\beta_2}$-modules that gives rise to an isomorphism of functors. The same holds for $\mathbb T_i$ if we consider functors from the category of ${}^i R_{s_i \beta_1} \boxtimes {}^i R_{s_i \beta_2}$-modules.
\end{thm}

\begin{rem}\label{ePBW}
{\rm 
Theorem \ref{main}, or rather its $\mathbb T$-version, corrects a mistake in the proof of \cite{Kat14} Lemma 4.2 2). Note that another correction was made for the arXiv version of \cite{Kat14}.}
\end{rem}

The rest of this section is devoted to the proof of Theorem \ref{main}, and the main body of the proof is at the end of this section.

Let $\beta_1, \beta _2 \in Q^+$ and set $\beta := \beta_1 + \beta_2$. The induction functor $\star$ is represented by a bimodule $R_{\beta} e _{\beta_1, \beta_2}$, where
$$e_{\beta_1, \beta_2} = \sum_{\mathbf m_1 \in Y^{\beta_1}, \mathbf m_2 \in Y^{\beta_2}} e ( \mathbf m_1 ) \boxtimes e ( \mathbf m_2 ).$$

We fix an orientation $\Omega$, and we might drop the superscript $\Omega$ freely if the meaning is clear from the context. We fix $I$-graded vector spaces $V(1)$ and $V ( 2 )$ so that $\underline{\dim} \, V ( i ) = \beta_i = \sum_{j \in I}d_j ( i ) \al_i$ for $i = 1,2$, and $V := V ( 1 ) \oplus V ( 2 )$.

We consider two varieties with natural $G_V$-actions:
\begin{align*}
\mathrm{Gr}^{\Omega}_{V(1),V(2)} ( V ) & := \Bigl\{ ( F, x, \psi_{1}, \psi _{2} ) \Bigl| {\footnotesize \begin{matrix} F \subset V \text{ : $I$-graded vector subspace} \\ x \in E_V, \text{ s.t. } x F \subset F \\\psi _1 : V/F \cong V(1), \psi_2 : F \cong V(2)  \end{matrix}} \Bigr\},\\
\mathrm{Gr}^{\Omega} _{\beta_1,\beta_2} ( V ) & := \Bigl\{ ( F, x ) \Bigl| {\footnotesize \begin{matrix} F \subset V \text{ : $I$-graded vector subspace} \\ x \in E_V, \text{ s.t. } x F \subset F \\ \underline{\dim} \, F = \beta _2  \end{matrix}} \Bigr\}.
\end{align*}
We have a $G_{V(1)} \times G_{V(2)}$-torsor structure $\vartheta^{\Omega} : \mathrm{Gr}_{V(1),V(2)} ( V ) \longrightarrow \mathrm{Gr} _{\beta_1,\beta_2} ( V )$ given by forgetting $\psi_1$ and $\psi _2$. We have two maps
\begin{align*}
\mathsf{p}^{\Omega} :  \, & \mathrm{Gr} _{\beta_1,\beta_2} ( V ) \ni ( F, x ) \mapsto x \in E _V \text{ and }\\
\mathsf{q}^{\Omega}  :  \, & \mathrm{Gr} _{V(1),V(2)} ( V ) \ni ( F, x, \psi_1, \psi_2) \mapsto ( \psi_1 ( x \!\!\! \mod F ), \psi_2 ( x \! \mid _F ) ) \in E_{V(1)} \oplus E_{V(2)}.
\end{align*}
Notice that $\vartheta$ and $\mathsf{q}$ are smooth of relative dimensions $\dim G_{V(1)} + \dim G _{V(2)}$ and $\frac{1}{2} ( \dim G_{V} + \dim G_{V(1)} + \dim G_{V(2)} ) + \sum _{h \in \Omega} d _1 ( h' ) d _2 ( h'' )$, respectively. The map $\mathsf{p}$ is projective. We set $N_{\beta_1,\beta_2} ^{\beta} := \frac{1}{2} ( \dim G_{V} - \dim G_{V(1)} - \dim G_{V(2)} ) + \sum _{h \in \Omega} d _1 ( h' ) d _2 ( h'' )$. For $G_{V(i)}$-equivariant constructible sheaves $\mathcal F_i$ on $E_{V(i)}$ for $i = 1,2$, we define their convolution product as
$$\mathcal F_1 \odot \mathcal F_2 := \mathsf{p} _! \mathcal F _{12} [N_{\beta_1,\beta_2} ^{\beta}], \text{ where } \vartheta^* \mathcal F_{12} \cong \mathsf{q} ^* ( \mathcal F_1 \boxtimes \mathcal F_2 ) \text{ in } D^b _{G_V} ( \mathrm{Gr}  _{V(1),V(2)} ( V ) ).$$

By construction, the convolution of $\mathcal L^{\Omega} _{\beta_1}$ and $\mathcal L^{\Omega} _{\beta_2}$ yields the direct summand of $\mathcal L^{\Omega} _{\beta}$ corresponding to the idempotent $e_{\beta_1, \beta_2}$. Hence, we have \begin{equation}
R_{\beta} e_{\beta_1,\beta_2} \cong \mathrm{Ext}^{\bullet} _{G_V} ( \mathcal L^{\Omega} _{\beta_1} \odot \mathcal L^{\Omega} _{\beta_2}, \mathcal L^{\Omega} _{\beta} )\label{Lconv}
\end{equation}
as $(R_{\beta},R_{\beta_1}\boxtimes R_{\beta_2})$-bimodule. 

Let $\mathcal L ^{\Omega}_{\beta_1, \beta_2}$ be a complex so that $\vartheta^* \mathcal L ^{\Omega}_{\beta_1, \beta_2} \cong \mathsf{q} ^* ( \mathcal L ^{\Omega}_{\beta_1} \boxtimes \mathcal L ^{\Omega} _{\beta_2} )$. Then, we have
\begin{align}\nonumber
\mathrm{Ext}^{\bullet} _{G_V} ( \mathcal L^{\Omega} _{\beta_1} \odot \mathcal L^{\Omega} _{\beta_2}, \mathcal L^{\Omega} _{\beta} ) & \cong \mathrm{Ext}^{\bullet} _{G_{V}} ( \mathcal L^{\Omega} _{\beta_1, \beta_2}, \mathsf{p}^! \mathcal L^{\Omega} _{\beta} ).\label{ind-eq}
\end{align}
Since we have
$$\mathrm{Ext}^{\bullet} _{G_{V}} ( \mathcal L^{\Omega} _{\beta_1, \beta_2}, \mathcal L^{\Omega} _{\beta_1, \beta_2} ) \cong \mathrm{Ext}^{\bullet} _{G_{V(1)} \times G_{V(2)}} ( \mathcal L^{\Omega} _{\beta_1} \boxtimes \mathcal L^{\Omega} _{\beta_2}, \mathcal L^{\Omega} _{\beta_1} \boxtimes \mathcal L^{\Omega} _{\beta_2} ),$$
we have a (right) $R_{\beta_1} \boxtimes R_{\beta_2}$-module structure of $R_{\beta} e_{\beta_1,\beta_2}$.

From now on, we assume that $i \in I$ is a sink of $\Omega$ and employ the setting of \S \ref{sSRF}. We find $\mathcal L ^{\Omega, \flat}_{\beta_1, \beta_2}$ so that
$$\vartheta^* \mathcal L ^{\Omega, \flat}_{\beta_1, \beta_2} \cong \mathsf{q} ^* ( ( j_{V(1)} )_! j_{V(1)}^! \mathcal L ^{\Omega}_{\beta_1} \boxtimes ( j_{V(2)} )_! j_{V(2)}^! \mathcal L ^{\Omega} _{\beta_2} ),$$
and $\mathcal O := \vartheta ( \mathsf{q}^{-1} ( {}_i E^{\Omega}_{V(1)} \times {}_i E^{\Omega}_{V(2)} ) )$. The graded vector space
$$\mathrm{Ext}^{\bullet} _{G_V} ( \mathsf{p}_! \mathcal L^{\Omega, \flat} _{\beta_1, \beta_2}, \mathcal L^{\Omega} _{\beta} ) \cong \mathrm{Ext}^{\bullet} _{G_{V}} ( \mathcal L^{\Omega, \flat} _{\beta_1, \beta_2}, \mathsf{p}^! \mathcal L^{\Omega} _{\beta} )$$
admits an $( R_{\beta}, {}_i R_{\beta_1} \boxtimes {}_i R_{\beta_2} )$-bimodule structure.

By restricting each components to the open set ${}_i E_V^{\Omega}$ by $j_V^* =j_V^!$, we deduce that $\mathrm{Ext}^{\bullet} _{G_V} ( j_V^* \mathsf{p}_! \mathcal L^{\Omega, \flat} _{\beta_1, \beta_2}, j_V^* \mathcal L^{\Omega} _{\beta} )$ is a left ${}_i R_{\beta}$-module. Applying adjunctions, this module is isomorphic to
\begin{equation}
\mathrm{Ext}^{\bullet} _{G_V} ( \mathsf{p}_! \mathcal L^{\Omega, \flat} _{\beta_1, \beta_2}, ( j_{V} )_* j_{V}^*\mathcal L^{\Omega} _{\beta} ) \cong \mathrm{Ext}^{\bullet} _{G_{V}} ( \mathcal L^{\Omega, \flat} _{\beta_1, \beta_2}, \mathsf{p}^! ( j_{V} )_* j_{V}^* \mathcal L^{\Omega} _{\beta} ),\label{pEisom}
\end{equation}
which admits a right ${}_i R_{\beta_1} \boxtimes {}_i R_{\beta_2}$-structure. Hence, (\ref{pEisom}) is an $( {}_i R_{\beta}, {}_i R_{\beta_1} \boxtimes {}_i R_{\beta_2} )$-bimodule.

We fix $I$-graded vector spaces $V'(1)$ and $V' ( 2 )$ so that $\underline{\dim} \, V' ( i ) = s_i \beta_i$ for $i = 1,2$. A similar construction as above implies that we have a sheaf $\mathcal L ^{s_i \Omega, \flat}_{s_i \beta_1, s_i \beta_2}$ so that
$$\vartheta^* \mathcal L ^{s_i \Omega, \flat}_{s_i \beta_1, s_i \beta_2} \cong \mathsf{q} ^* ( ( \jmath _{V'(1)} )_! \jmath _{V'(1)}^! \mathcal L ^{s_i \Omega}_{s_i \beta_1} \boxtimes ( \jmath _{V'(2)} )_! \jmath _{V'(2)}^! \mathcal L ^{s_i \Omega} _{s_i \beta_2} ).$$
It yields an $( {}^i R_{s_i \beta}, {}^i R_{s_i \beta_1} \boxtimes {}^i R_{s_i \beta_2} )$-bimodule
\begin{equation}
\mathrm{Ext}^{\bullet} _{G_{V'}} ( \mathsf{p}_! \mathcal L^{s_i \Omega, \flat} _{s_i \beta_1, s_i \beta_2}, ( \jmath_{V'} )_* \jmath_{V'}^*\mathcal L^{s_i \Omega} _{s_i \beta} ) \cong \mathrm{Ext}^{\bullet} _{G_{V'}} ( \mathcal L^{s_i \Omega, \flat} _{s_i \beta_1, s_i \beta_2}, \mathsf{p}^! ( \jmath_{V'} )_* \jmath_{V'}^* \mathcal L^{s_i \Omega} _{s_i \beta} ).
\end{equation}

\begin{thm}\label{restmain}
Under the above setting, the image of the natural restriction map
$$\mathrm{Ext}^{\bullet} _{G_V} ( \mathsf{p}_! \mathcal L^{\Omega} _{\beta_1, \beta_2}, \mathcal L^{\Omega} _{\beta} ) \longrightarrow \mathrm{Ext}^{\bullet} _{G_V} ( \mathsf{p}_! \mathcal L^{\Omega, \flat} _{\beta_1, \beta_2}, \mathcal L^{\Omega} _{\beta} )$$
is a submodule of the RHS, and is equal to
$${}_i R_{\beta} \otimes_{R_{\beta_1} \boxtimes R_{\beta_2}} ( {}_i R_{\beta_1} \boxtimes {}_i R_{\beta_2} ).$$
In addition, it is the pure part of weight zero in $\mathrm{Ext}^{\bullet} _{G_V} ( \mathsf{p}_! \mathcal L^{\Omega, \flat} _{\beta_1, \beta_2}, \mathcal L^{\Omega} _{\beta} )$. The same is true if we replace $\Omega$ with $s_i \Omega$, $\beta_j$ by $s_i \beta_j$, and ${}_i R_{\beta_j}$ with ${}^i R_{s_i \beta_j}$  $(j = \emptyset, 1,2)$.
\end{thm}

\begin{proof}
Since the proofs of the both assertions are similar, we prove only the case of $\Omega$. We have $\mathcal O \subset \mathsf{p}^{-1} ({}_i E^{\Omega}_V )$, and hence the restriction map factors through the restriction to ${}_i E^{\Omega}_V$. By unwinding the definition, we have a factorization
\begin{align*}
\mathrm{Ext}^{\bullet} _{G_V} ( \mathsf{p}_! \mathcal L^{\Omega} _{\beta_1, \beta_2}, \mathcal L^{\Omega} _{\beta} ) & \longrightarrow \!\!\!\!\! \rightarrow \mathrm{Ext}^{\bullet} _{G_V} ( j_V^! \mathsf{p}_! \mathcal L^{\Omega} _{\beta_1, \beta_2}, j_V^! \mathcal L^{\Omega} _{\beta} ) \\
\cong & \mathrm{Ext}^{\bullet} _{G_V} ( ( j_V )_! j_V^! \mathsf{p}_! \mathcal L^{\Omega} _{\beta_1, \beta_2}, \mathcal L^{\Omega} _{\beta} ) \stackrel{\rho}{\longrightarrow} \mathrm{Ext}^{\bullet} _{G_V} ( \mathsf{p}_! \mathcal L^{\Omega, \flat} _{\beta_1, \beta_2}, \mathcal L^{\Omega} _{\beta} )
\end{align*}
of $( R_{\beta}, R_{\beta_1} \boxtimes R_{\beta_2} )$-bimodule map, where the first map (that is surjection by Theorem \ref{i-quotients}) is the restriction to the open set, the second isomorphism is the adjunction, and the third morphism is obtained by the base change using $j_V^! = j_V^*$ and the composition.

We have a distinguished triangle
$$( j_{V(2)} )_! j_{V(2)}^! \mathcal L^{\Omega} _{\beta_2} \rightarrow \mathcal L^{\Omega} _{\beta_2} \rightarrow \mathsf{Ker} \stackrel{+1}{\rightarrow}.$$
Since $\mathcal L^{\Omega} _{\beta_2}$ is pure of weight $0$, it follows that $( j_{V(2)} )_! j_{V(2)}^! \mathcal L^{\Omega} _{\beta_2}$ must have weight $\le 0$ (\cite[5.1.14]{BBD}). Taking account into the fact that $( j_{V(2)} )_! j_{V(2)}^! \mathcal L^{\Omega} _{\beta_2}$ and $\mathcal L^{\Omega} _{\beta_2}$ share the same stalk along ${}_i E^{\Omega}_{V(2)}$ and the stalk of $( j_{V(2)} )_! j_{V(2)}^! \mathcal L^{\Omega} _{\beta_2}$ vanishes outside of ${}_i E^{\Omega}_{V(2)}$, we conclude that $\mathsf{Ker}$ has weight $\le 0$. We set $\mathcal K := \mathsf{p}_! \mathcal K'$, where $\vartheta^* \mathcal K' \cong \mathsf{q}^* ( \mathcal L^{\Omega} _{\beta_1} \boxtimes \mathsf{Ker} )$.

From now on, we make all computations over ${}_i E_V^{\Omega}$ by using $j_V^* = j_V^!$. The above construction gives us a distinguished triangle
$$\mathsf{p}_! \mathcal L^{\Omega, \flat} _{\beta_1, \beta_2} \rightarrow \mathcal L^{\Omega} _{\beta_1} \odot \mathcal L^{\Omega} _{\beta_2} \rightarrow \mathcal K \stackrel{+1}{\longrightarrow}.$$
Moreover, $\mathcal K$ has weight $\le 0$ by $\mathsf{p}_* = \mathsf{p}_!$.

Hence, we deduce an exact sequence of ${}_i R_{\beta}$-modules
$$\mathrm{Ext}^{\bullet} _{G_V} ( \mathcal K, \mathcal L^{\Omega} _{\beta} ) \rightarrow \mathrm{Ext}^{\bullet} _{G_V} ( \mathsf{p}_! \mathcal L^{\Omega} _{\beta_1, \beta_2}, \mathcal L^{\Omega} _{\beta} ) \stackrel{\rho}{\longrightarrow} \mathrm{Ext}^{\bullet} _{G_V} ( \mathsf{p}_! \mathcal L^{\Omega, \flat} _{\beta_1, \beta_2}, \mathcal L^{\Omega} _{\beta} ).$$
Note that the middle term has weight $0$ by Theorem \ref{i-quotients} as the both of $\mathcal L^{\Omega} _{\beta_1} \odot \mathcal L^{\Omega} _{\beta_2}$ and $\mathcal L^{\Omega} _{\beta}$ are pure of weight $0$. Since $\mathrm{Ext}^{\bullet} _{G_V} ( \mathcal K, \mathcal L^{\Omega} _{\beta} )$ has weight $\ge 0$ (\cite[5.1.14]{BBD}), we conclude that $\mathrm{Im} \, \rho$ is precisely the weight $0$-part of $\mathrm{Ext}^{\bullet} _{G_V} ( \mathsf{p}_! \mathcal L^{\Omega, \flat} _{\beta_1, \beta_2}, \mathcal L^{\Omega} _{\beta} )$ (see also the arguments in \cite{Kat17}).

Since the $( {}_i R_{\beta}, {}_i R_{\beta_1} \boxtimes {}_i R_{\beta_2})$-action preserves the weight, it follows that $\mathrm{Im} \, \rho$ is an $( {}_i R_{\beta}, {}_i R_{\beta_1} \boxtimes {}_i R_{\beta_2})$-subbimodule of $\mathrm{Ext}^{\bullet} _{G_V} ( \mathsf{p}_! \mathcal L^{\Omega, \flat} _{\beta_1, \beta_2}, \mathcal L^{\Omega} _{\beta} )$. Since we have $\mathrm{Ext}^{\bullet} _{G_V} ( \mathsf{p}_! \mathcal L^{\Omega} _{\beta_1, \beta_2}, \mathcal L^{\Omega} _{\beta} ) \cong {}_i R_{\beta} e_{\beta_1,\beta_2}$, we have a surjection 
$$\pi : {}_i R_{\beta} e_{\beta_1,\beta_2} \longrightarrow \!\!\!\!\! \rightarrow \mathrm{Im} \, \rho.$$

By Proposition \ref{i-Res}, the sheaf $\mathsf{Ker}$ is obtained by successive constructions of cones of shifted perverse sheaves on $\mathcal Q_{\beta_2}^{\Omega}$ that are supported outside of ${}_i E_{V( 2)}^{\Omega}$. Therefore, we deduce that $\ker \, \rho$ admits a surjection from the direct sum of $R_{\beta}$-modules of the form
$$P_{b_1} \star P_{b_2} \hskip 5mm b_1 \in B ( \infty )_{\beta_1}, b_2 \in B ( \infty )_{\beta_2}, \epsilon_i ( b_2 ) > 0,$$
that corresponds to $\mathsf{IC}^{\Omega} ( b_{1} ) \odot \mathsf{IC}^{\Omega} ( b_{2} )$ with $\epsilon_i ( b_2 ) > 0$. Let us write $E$ the sum of the image of all such $R_{\beta}$-modules in ${}_i R_{\beta} e_{\beta_1,\beta_2}$ arises as the above induction. In view of the construction of $\mathsf{Ker}$, we have $\ker \, \pi \subset E$.

On the other hand, $E$ is precisely the kernel of the natural quotient map
$${}_i R_{\beta} e_{\beta_1,\beta_2} \longrightarrow \!\!\!\!\! \rightarrow {}_i R_{\beta} \otimes_{R_{\beta_1} \boxtimes R_{\beta_2}} ( {}_i R_{\beta_1} \boxtimes {}_i R_{\beta_2} ).$$

As a consequence, we have a quotient map
$$\mathrm{Im}\, \rho\longrightarrow \!\!\!\!\!\rightarrow {}_i R_{\beta} \otimes_{R_{\beta_1} \boxtimes R_{\beta_2}} ( {}_i R_{\beta_1} \boxtimes {}_i R_{\beta_2} ).$$

The module $\mathrm{Im} \, \rho$ is a $( {}_i R_{\beta}, {}_i R_{\beta_1} \boxtimes {}_i R_{\beta_2})$-bimodule whose bimodule structure is induced from the $( {}_i R_{\beta}, R_{\beta_1} \boxtimes R_{\beta_2})$-bimodule structure on ${}_i R_{\beta}  e_{\beta_1,\beta_2}$ by construction (through Theorem \ref{i-quotients}). Thus, $\mathrm{Im} \, \rho$ admits a surjection from ${}_i R_{\beta} \otimes_{R_{\beta_1} \boxtimes R_{\beta_2}} ( {}_i R_{\beta_1} \boxtimes {}_i R_{\beta_2} )$, that is the maximal $( {}_i R_{\beta}, {}_i R_{\beta_1} \boxtimes {}_i R_{\beta_2})$-bimodule quotient of ${}_i R_{\beta}  e_{\beta_1,\beta_2}$ (regarded as a $( {}_i R_{\beta}, R_{\beta_1} \boxtimes R_{\beta_2})$-bimodule). Therefore, we conclude
$$\mathrm{Im} \, \rho \cong {}_i R_{\beta} \otimes_{R_{\beta_1} \boxtimes R_{\beta_2}} ( {}_i R_{\beta_1} \boxtimes {}_i R_{\beta_2} )$$
as required.
\end{proof}

\begin{proof}[Proof of Theorem \ref{main}]
Note that the open subset $\mathcal O \subset \mathrm{Gr}_{\beta_1,\beta_2} ^{\Omega} ( V )$ is precisely the set of points $( F, x )$ so that $x \! \mid_F \in {}_iE_{V(2)}^{\Omega}$ and $x \mod F \in {}_i E_{V(1)}^{\Omega}$. We set $\mathcal O' := \vartheta ( \mathsf{q}^{-1} ( {}^i E^{s_i \Omega}_{V'(1)} \times {}^i E^{s_i \Omega}_{V'(2)} ) )$. The open subset $\mathcal O' \subset \mathrm{Gr}_{s_i\beta_1,s_i\beta_2} ^{s_i \Omega} ( V' )$ is precisely the set of points $( F', x' )$ so that $x' \! \mid_{F'} \in {}^i E_{V'(2)}^{s_i \Omega}$ and $x' \mod F' \in {}^i E_{V'(1)}^{s_i \Omega}$.

Therefore, (\ref{df}) yields a variety $\mathbb O$ with the $G_{V,V'}$-action defined as:
$$\Biggl\{ \{ \{W_i, W'_i\}_{i}, ( f _h ) _{h}, ( f' _h ) _{h}, \phi, \psi \} \Biggl| {\small \begin{matrix} \{ ( f _h ) _{h \in \Omega}, ( f' _h ) _{h \in s_i \Omega}, \psi \} \in Z^{\Omega}_{V,V'},\\
\phi : W_j \cong W'_j \text{ for } j \neq i\\
( \{W_i\}_{i \in I}, ( f _h ) _{h \in \Omega} ) \in \mathrm{Gr}_{\beta_1,\beta_2} ^{\Omega} ( V ), \\
( \{W'_i\}_{i \in I}, ( f' _h ) _{h \in s_i \Omega} ) \in \mathrm{Gr}_{s_i \beta_1,s_i\beta_2} ^{s_i\Omega} ( V' )\\ \psi : W_i' \stackrel{\cong}{\longrightarrow} \mathrm{ker} ( \bigoplus _{h \in \Omega_i} f_h : \bigoplus _h W_{h'} \to W_i )\end{matrix}} \Biggr\}.$$
Note that the condition $( f _h \! \mid_{\{W_i\}_{i}} ) _{h \in \Omega} \in {}_i E_{V(2)}^{\Omega}$ guarantees that
$$\dim \, W_i' = \dim \, \mathrm{ker} ( \bigoplus _{h \in \Omega_i} f_h : \bigoplus _h W_{h'} \to W_i )$$
and similarly the condition $( f' _h \! \mid_{\{W_i\}_{i}} ) _{h \in s_i \Omega} \in {}^i E_{V'(2)}^{s_i \Omega}$ guarantees that
$$\dim \, W_i = \dim \, \mathrm{coker} ( \bigoplus _{h \in \Omega_i} f_h : W'_i \to \bigoplus _h W_{h'} ),$$
that actually asserts the same thing. Since we have an isomorphism
$$\psi : V_i' \stackrel{\cong}{\longrightarrow} \mathrm{ker} ( \bigoplus _{h \in \Omega_i} f_h : \bigoplus _h V_{h'} \to V_i )$$
from the definition of $Z^{\Omega}_{V,V'}$, taking quotients yield
$$( f _h \! \mod {\{W_i\}_{i}} ) _{h \in \Omega} \in {}_i E_{V(1)}^{\Omega} \hskip 3mm \text{and} \hskip 3mm ( f _h \! \mod {\{W'_i\}_{i}} ) _{h \in s_i \Omega} \in {}^i E_{V'(1)}^{s_i \Omega}.$$

Hence, the quotients of $\mathbb O$ by $G_{V'_i}$ and $G_{V_i}$ gives $\widetilde{q}^i_V$ and $\widetilde{p}^i_{V'}$ in the commutative diagram in the below:
$$
\xymatrix{
\mathrm{Gr} _{\beta_1,\beta_2}^{\Omega} ( V ) \ar[d]_{\mathsf{p}^{\Omega}} & \mathcal O \ar@{_{(}->}[l] \ar[d]^{p} & \mathbb O \ar@{->>}[r]^{\widetilde{p}^i_{V'}} \ar@{->>}[l] _{\widetilde{q}^i _{V}} & \mathcal O'  \ar[d]_{p'} \ar@{^{(}->}[r]& \mathrm{Gr} _{s_i \beta_1,s_i \beta_2}^{s_i\Omega} ( V' ) \ar[d]_{\mathsf{p}^{s_i \Omega}}\\
E ^{\Omega} _V & {}_i E ^{\Omega} _V \ar@{_{(}->}[l] _{j_V}& Z _{V,V'}^{\Omega} \ar@{->>}[r]^{p^i_{V'}} \ar@{->>}[l] _{q^i _{V}} & {}^i E ^{s_i \Omega} _{V'} \ar@{^{(}->}[r] ^{\jmath _{V'}}& E ^{s_i \Omega} _{V'} \hskip 5mm.}
$$
Therefore, we have an equivalence of the category of $G_V$-equivariant sheaves on $\mathcal O$, and the category of $G_{V'}$-equivariant sheaves on $\mathcal O'$ (cf. \cite[\S 2.6.3]{BL94}). With an aid of Proposition \ref{TM}, we conclude that
$$\mathrm{Ext}^{\bullet} _{G_{V}} ( \mathcal L^{\Omega, \flat} _{\beta_1, \beta_2}, p^! \mathcal L^{\Omega} _{\beta} ) \cong \mathrm{Ext}^{\bullet} _{G_{V'}} ( \mathcal L^{s_i \Omega, \flat} _{s_i \beta_1, s_i \beta_2}, (p')^! \mathcal L^{s_i \Omega} _{s_i \beta} )$$
up to amplifications of direct summands (i.e. we allow to duplicate direct summand of both terms). By Theorem \ref{restmain}, the comparison of their weight zero parts identifies
$${}_i R_{\beta} \otimes_{R_{\beta_1} \boxtimes R_{\beta_2}} ( {}_i R_{\beta_1} \boxtimes {}_i R_{\beta_2} ) \text{
 and } {}^i R_{s_i \beta} \otimes_{R_{s_i\beta_1} \boxtimes R_{s_i\beta_2}} ( {}^i R_{s_i\beta_1} \boxtimes {}^i R_{s_i \beta_2} )$$
through the Morita equivalences in Proposition \ref{TM}. This is actually an identification of bimodules by construction.

In other words, we have an isomorphism
$$\mathbb T_i^* ( \mathrm{Ext}^{\bullet} _{G_V} ( j_V^* ( \mathcal L^{\Omega, \flat} _{\beta_1, \beta_2} ), j_V^* \mathcal L^{\Omega} _{\beta} ) ) \cong \mathrm{Ext}^{\bullet} _{G_{V'}} ( \jmath_{V'}^* ( \mathcal L^{s_i \Omega, \flat} _{s_i \beta_1, s_i\beta_2} ), \jmath_{V'}^* \mathcal L^{s_i \Omega} _{s_i \beta} ),$$
where the amplification of direct summands is subsumed in the constructions of $\mathbb T_i$. This isomorphism commutes with the Morita equivalence of ${}_i R_{\beta_j}$ and ${}^i R_{s_i \beta_j}$ for $j = 1,2$ by the above. Hence, taking their weight $0$ part yields the desired natural transformation
$$\mathbb T_i^* ( \bullet \star \bullet ) \longrightarrow \mathbb T_i^* ( \bullet ) \star \mathbb T_i^* ( \bullet )$$
of functors, and it must be an equivalence. The case of $\mathbb T_i$ is obtained similarly.
\end{proof}

{\small
\hskip -5.25mm {\bf Acknowledgement:} The author would like to thank Masaki Kashiwara and Myungho Kim for helpful correspondences. He also thanks Peter McNamara who kindly sent me a version of \cite{M17b}. This research is supported in part by JSPS Grant-in-Aid for Scientific Research (B) JP26287004.}

{\footnotesize
\bibliography{mrref}

\begin{thebibliography}{10}

\bibitem{BBD}
A.~A. Beilinson, J.~Bernstein, and P.~Deligne.
\newblock {\em Faisceaux pervers}, volume 100 of {\em Ast\'erisque}.
\newblock Soc. Math. France, Paris, 1982.

\bibitem{BL94}
Joseph Bernstein and Valery Lunts.
\newblock {\em Equivariant sheaves and functors}, volume 1578 of {\em Lecture
  Notes in Mathematics}.
\newblock Springer-Verlag, Berlin, 1994.

\bibitem{CG97}
Neil Chriss and Victor Ginzburg.
\newblock {\em Representation theory and complex geometry}.
\newblock Modern Birkh\"auser Classics. Birkh\"auser Boston, Inc., Boston, MA,
  2010.
\newblock Reprint of the 1997 edition.

\bibitem{Kas91}
Masaki Kashiwara.
\newblock {On crystal bases of the {\$}q{\$}-analogue of universal enveloping
  algebras}.
\newblock {\em Duke Mathematical Journal}, 63(2):465--516, 1991.

\bibitem{KKOP17}
Masaki Kashiwara, Myungho Kim, Se~jin Oh, and Euiyong Park.
\newblock Monoidal categories associated with strata of flag manifolds.
\newblock {\em Adv. in Math.}, 328:959--1009, 2018.
\newblock arXiv:1708.04428.

\bibitem{Kat14}
Syu Kato.
\newblock Poincar\'e-{B}irkhoff-{W}itt bases and {K}hovanov-{L}auda-{R}ouquier
  algebras.
\newblock {\em Duke Math. J.}, 163(3):619--663, 2014.
\newblock Correction abridged version: arXiv:1203.5254v5.

\bibitem{Kat17}
Syu Kato.
\newblock An algebraic study of extension algebras.
\newblock {\em Amer. J. Math.}, 139(3):567--615, 2017.

\bibitem{KL09}
M.~Khovanov and A.~Lauda.
\newblock A diagrammatic approach tocategorification of quantum groups {I}.
\newblock {\em Represent. Theory}, 13:309--347, 2009.

\bibitem{LV11}
Aaron~D. Lauda and Monica Vazirani.
\newblock Crystals from categorified quantum groups.
\newblock {\em Adv. Math.}, 228(2):803--861, 2011.

\bibitem{Lus90a}
G.~Lusztig.
\newblock Canonical bases arising from quantized enveloping algebras.
\newblock {\em J. Amer. Math. Soc.}, 3(2):447--498, 1990.

\bibitem{Lus93}
G.~Lusztig.
\newblock {\em Introduction to quantum groups}, volume 110 of {\em Progress in
  Mathematics}.
\newblock Birkh\"auser Boston, Inc., Boston, MA, 1993.

\bibitem{Lus98}
G.~Lusztig.
\newblock Canonical bases and {H}all algebras.
\newblock In {\em Representation theories and algebraic geometry (Montreal, PQ,
  1997)}, NATO Adv. Sci. Inst. Ser. C Math. Phys. Sci., 514, pages 365--399.
  Kluwer Acad. Publ., 1998.

\bibitem{Lus91a}
George Lusztig.
\newblock Quivers, perverse sheaves, and quantized enveloping algebras.
\newblock {\em Journal of the American Mathematical Society}, 4(2):365--421,
  1991.

\bibitem{M17}
Peter~J. McNamara.
\newblock Representation theory of geometric extension algebras.
\newblock arXiv:1701.07949.

\bibitem{M17b}
Peter~J. McNamara.
\newblock Monoidality of {K}ato's reflection functors.
\newblock arXiv:1712.00173, 2017.

\bibitem{R}
R~Rouquier.
\newblock 2-{K}ac-{M}oody algebras.
\newblock arXiv:0812.5023.

\bibitem{VV11}
Michela Varagnolo and Eric Vasserot.
\newblock {Canonical bases and KLR-algebras}.
\newblock {\em Journal f{\"{u}}r die reine und angewandte Mathematik (Crelles
  Journal)}, 659:67--100, 2011.

\bibitem{XZ17}
Jie Xiao and Minghui Zhao.
\newblock Geometric realizations of {L}usztig's symmetries.
\newblock {\em J. Algebra}, 475:392--422, 2017.

\end{thebibliography}
\bibliographystyle{hplain}}
\end{document}